\documentclass[12pt,reqno]{amsart}
\usepackage{hyperref}
\usepackage{cleveref}
\usepackage{amsmath,amssymb}
\usepackage[english]{babel}
\usepackage{caption}
\usepackage{amssymb,amsmath,euscript,enumerate,tikz}
\usepackage[margin=1in]{geometry}
\usepackage{graphicx}
\usepackage{float}
\usepackage{subcaption}

\usepackage{pgf,tikz}
\usepackage{mathrsfs}
\usepackage{upgreek}
\usepackage[normalem]{ulem}

\newcommand \reg{\operatorname{reg}}
\newcommand\sing{\operatorname{Sing}}

\newcommand \pd{\operatorname{pdim}}

\newcommand \cmdef{\operatorname{cmdef}}

\newcommand \depth{\operatorname{depth}}

\newcommand \LL{\mathcal{L}}

\newcommand \KB{\mathbb{K}}

\newtheorem{theorem}{Theorem}[section]
\newtheorem{definition}[theorem]{Definition}
\newtheorem{lemma}[theorem]{Lemma}

\newtheorem{example}[theorem]{Example}

\newtheorem{question}[theorem]{Question}
\newtheorem{corollary}[theorem]{Corollary}

\newtheorem{con}[theorem]{Conjecture}

\newtheorem{remark}[theorem]{Remark}

\begin{document}
\title[binomial edge ideals of Levi graphs associated with curve arrangements]{Algebraic properties of binomial edge ideals of Levi graphs associated with curve arrangements}
\author[Rupam Karmakar]{Rupam Karmakar}
\email{rupammath91@gmail.com}
\author[Rajib Sarkar]{Rajib Sarkar}
\email{rajib.sarkar63@gmail.com}
\address{Stat-Math Unit, Indian Statistical Institute \newline \indent 203 B.T. Road, Kolkata--700108, India.}

\author[Aditya Subramaniam]{Aditya Subramaniam}
\address{Indian Institute of Science Education and Research Tirupati, Rami Reddy Nagar, Karakambadi Road, Mangalam (P.O.),
Tirupati, Andhra Pradesh, INDIA – 517507.}
\email{adityasubramaniam@labs.iisertirupati.ac.in / adisubbu92@gmail.com}

\begin{abstract}
In this article, we study algebraic properties of binomial edge ideals of Levi graphs associated with certain plane curve arrangements. Using combinatorial properties of Levi graphs, we discuss the Cohen-Macaulayness of binomial edge ideals of Levi graphs associated to some curve arrangements in the complex projective plane, like the $d$-arrangement of curves and the conic-line arrangements. We also discuss the existence of certain induced cycles in the Levi graphs of these arrangements and obtain lower bounds for the regularity of powers of the corresponding binomial edge ideals.

\end{abstract}

\maketitle
\section{Introduction}

In the present article, we study the binomial edge ideals of Levi graphs associated to certain curve arrangements in the complex projective plane. We briefly recall the notion of binomial edge ideals corresponding to simple graphs below.

Let $G$ be a  simple graph with the vertex set $V(G) =[n]:= \{1, \ldots, n\}$ and the edge set $E(G).$ Let $S=
\KB[x_1, \ldots, x_{n}, y_1, \ldots, y_{n}]$ be the polynomial ring in $2n$ variables over an arbitrary field $\KB$. The binomial edge ideal of $G$, denoted by $J_G$,  is defined as 
$$J_G:=\langle x_i y_j - x_j y_i ~ : i < j \text{ and } \{i,j\}\in E(G) \rangle \subseteq S.$$ 

Binomial edge ideals were introduced by Herzog et al. \cite{HHHKR} in the study of conditional independence statements in Algebraic Statistics. These ideals were also independently introduced by Ohtani in \cite{oh} as an ideal generated by certain $2$-minors of a $(2 \times n)$-generic matrix corresponding to the edges of a graph on $n$ vertices.

There are several interesting directions in which binomial edge ideals are being studied. See \cite{Priya-survey,Madani-survey} for a comprehensive survey. One of the most important problems in the study of binomial edge ideals is to find a characterization of Cohen-Macaulay binomial edge ideals. For a sampling of the many results in this direction, see \cite{BMRS22,DAV,DAV2,EHH-NMJ,KM-CA,Rinaldo-BMS,Rinaldo-Cactus,RS-level,SS22}.

In this article, we continue this study by looking at binomial edge ideals of Levi graphs associated to some plane curve arrangements and understand their algebraic properties like Cohen-Macaulayness and regularity. Our primary motivation is \cite{PR22}, where homological properties of edge ideals of Levi graphs associated with certain plane curve arrangements are studied. We briefly recall the notion of Levi graphs associated with plane curve arrangements below.

Levi graphs are bipartite graphs that are naturally associated to plane curve arrangements with ordinary singularities. These graphs were introduced by Coxeter in \cite{Cox} and have been primarily studied in the case of line arrangements. For example, if we had a collection of points and lines in the projective plane, the associated Levi graph would be a bipartite graph with one vertex per point, one vertex per line, and an edge for every incidence between a point and a line. Levi graphs are important as they encode various properties of intersection posets  of curve arrangements with ordinary singularities. For example, these graphs play a central role in studying the freeness of line arrangements, as seen in the following well-known Terao's conjecture \cite{T092,HT81}:
\begin{con}(Terao's Conjecture)
 Let $\mathcal{C}, \mathcal{C}' \subset \mathbb{P}^2_{\mathbb{C}}$ be line arrangements and $G_{1},G_{2}$ be the associated Levi graphs. Assume that $G_1$ and $G_2$ are isomorphic. Then, if $\mathcal{C}'$ is free, then  $\mathcal{C}$ is free. 
\end{con}

\vspace{2mm}

In this article, we consider Levi graphs of certain arrangements of plane curves of a fixed degree $d,$ known in the literature as $d$-arrangements and study their binomial edge ideals. These arrangements  are generalizations of line arrangements and first appeared in \cite{PRS} in the context of the bounded negativity conjecture. 
Levi graphs of $d$-arrangements were considered in \cite{PR22}, where  the authors show that most of the edge ideals of Levi graphs associated with $d$-arrangements are never Cohen-Macaulay. We prove an analogous result in the case of binomial edge ideals.

We also study binomial edge ideals of Levi graphs associated to conic-line arrangements with ordinary singularities in the complex projective plane. These arrangements were considered in  \cite{PSCo} in studying Harbourne constants and geography problem for log-surfaces.

Another homological property of interest in the study of homogeneous ideals in polynomial rings is the regularity of its powers. For edge ideals of Levi graphs associated to certain $d$-arrangements, Pokora and R\"{o}mer \cite{PR22} give bounds for the regularity of their powers. Motivated by their work, we give lower bounds for the regularity of powers of binomial edge ideals of Levi graphs associated to certain $d$-arrangements and conic-line arrangements by showing the existence of certain induced cycles in the associated Levi graphs.

The paper is organized as follows.

In Section \ref{Prelim}, we recall some preliminaries and notations that are used in later sections. In Section \ref{Sec: CM}, we study the Cohen-Macaulayness of binomial edge ideals of Levi graphs coming from certain plane curve arrangements. Firstly, we give an example of a point-line configuration such that the binomial edge ideal of the associated Levi graph is Cohen-Macaulay. Then, in Section \ref{subs-hir}, we study the binomial edge ideals of Levi graphs associated to Hirzebruch quasi-pencils and prove that they are never Cohen-Macaulay. Moreover, we compute the dimension of these binomial edge ideals and give a lower bound on their Cohen-Macaulay defect.

In Section \ref{sec: not-Cohen-Macaulay}, we show that  most of the binomial edge ideals of Levi graphs associated with $d$-arrangements and conic-line arrangements in the complex projective plane are never Cohen-Macaulay.

In Section \ref{Sec: Prop}, we prove that the Levi graphs of certain $d$-arrangements and conic-line arrangements have an induced $\mathbf{C}_6$ and give examples of $d$-arrangements whose Levi graphs have an induced cycle of maximum length. Using these results in Section \ref{Regu}, we obtain bounds for the regularity of powers of binomial edge ideals of Levi graphs coming from $d$-arrangements and conic-line arrangements.

\section{Acknowledgment}
The first and second authors thank the National Board for Higher Mathematics (NBHM), Department of Atomic Energy, Government of India for financial support through their postdoctoral fellowship. We are thankful to the anonymous referees for carefully reading the manuscript and making several suggestions that improved the exposition.
\section{Preliminaries}\label{Prelim}
\subsection{ Plane Curve arrangements and Levi graphs}
We recall some basic notations, definitions and known results that will be used throughout this article.

Firstly, we recall the notion of a $d$-arrangement of curves in the complex projective plane.

\begin{definition}
Let $\mathcal{C} = \{C_{1}, \dots , C_{k}\} \subset \mathbb{P}^{2}_{\mathbb{C}}$ be an arrangement of $k\geq 3$ curves in the complex projective plane. We say that $\mathcal{C}$ is a  \emph{$d$-arrangement} if
\begin{itemize}
\item all curves $C_{i}$ are smooth of the same degree $d\geq 1$;
\item the singular locus ${\rm Sing}(\mathcal{C})$ consists of only ordinary intersection points -- they look locally like intersections of lines.
\end{itemize}
\end{definition}

In particular, 1-arrangements are line arrangements and 2-arrangements are conic arrangements with ordinary intersection points. We have the following combinatorial count for a $d$-arrangement $\mathcal{C}$ :
\begin{align}
    d^2 {k \choose 2} = \sum_{p \in \sing(\mathcal{C})} {m_p \choose 2},
\end{align}
where $m_p$ denotes the multiplicity at $p$ i.e., the number of curves in $\mathcal{C}$ passing through $p \in \sing(\mathcal{C}).$ 
Also, we have for every $C_i\in \mathcal{C}$, 
\begin{align}\label{comb count}
    d^2(k-1)=\sum_{p\in \sing(\mathcal{C})\cap C_i}(m_p-1).
\end{align}
Let $t_r(\mathcal{C})$ be the number of $r$-fold points in $\sing{(\mathcal{C})}$ i.e., the number of points where exactly $r$ curves from $\mathcal{C}$ meet. So, if $s(\mathcal{C})$ is the total number of intersection points in $\mathcal{C}$, then $s(\mathcal{C})= \sum_{r \geq 2}t_r(\mathcal{C})$. We sometimes write $s, t_r$ in place of $s(\mathcal{C}), t_r(\mathcal{C})$ respectively if there is no confusion about $\mathcal{C}$. 
\vspace{2mm}

We now recall the notion of Levi graphs for $d$-arrangements.

\begin{definition}\label{df: d-arr}
Let $\mathcal{C}=\{C_{1}, \dots , C_{k}\} \subset \mathbb{P}^{2}_{\mathbb{C}}$ be a $d$-arrangement with $|\sing(\mathcal{C})|=s$. Then the associated \emph{Levi graph} $G = (V,E)$ is a bipartite graph with  $V : = V_{1} \cup V_{2} = \{x_{1}, \dots , x_{s}, y_{1}, \dots , y_{k}\}$,  where each vertex $x_i$ corresponds to an intersection point $p_i \in \sing(\mathcal{C})$, each vertex $y_j$ corresponds to the curve $C_j$ in $\mathcal{C}$ and vertices $x_i$, $y_j$ are joined by an edge in $E$ if and only if $p_i$ is incident with $C_j$. 
\end{definition} 

We next recall the notion of conic-line arrangements in the complex projective plane. These are special types of curve arrangements where all the curves do not have the same degree. 
\begin{definition}
    A conic-line arrangement $\mathcal{CL} = \{\ell_{1}, \dots ,\ell_{n}, C_{1}, \dots , C_{k}\} \subset \mathbb{P}^{2}_{\mathbb{C}}$ is an arrangement of $n$ lines and $k$ conics having only ordinary singularities i.e., the intersection points look locally as $\{x^{a}=y^{a}\}$ for some integer $a\geq 2.$
\end{definition}
By B\'ezout's theorem, we have the following combinatorial count:
\begin{align}
    4 {k \choose 2} + {n \choose 2} + 2kn = \sum_{r \geq 2}{r \choose 2}t_r,
\end{align}
where $t_r$ is the number of $r$-fold points in $\sing{(\mathcal{CL})}$ i.e., the number of points where exactly $r$ curves from $\mathcal{CL}$ meet.

Similar to the $d$-arrangement case, we can also associate a Levi graph to each conic-line arrangement.
\begin{definition}\label{df: conic line-arr}
Let $\mathcal{CL} = \{\ell_{1}, \dots ,\ell_{n}, C_{1}, \dots , C_{k}\} \subset \mathbb{P}^{2}_{\mathbb{C}}$ be a conic-line arrangement with $|\sing(\mathcal{CL})|=s$. Then the associated  \emph{Levi graph} $H = (W, F)$ is a bipartite graph with $W := W_1 \cup W_2 = \{x_1,\dots ,x_s, y_1,\dots , y_{n+k}\}$, where each vertex $x_i$ corresponds to an intersection point $p_i \in \sing(\mathcal{CL})$, each vertex $y_j$ corresponds to the line $\ell_j$ for $j=1,\dots,n$, each vertex $y_{n+j}$ corresponds to the conic $C_j$ for $j = 1,\dots,k$ in $\mathcal{CL}$ and vertices $x_i$,$y_j$ or $x_i, y_{n+j}$ are joined by an edge in $F$ if and only if  $p_i$ is incident with $\ell_j$ or $C_j$  for some $j.$ 
\end{definition} 

\subsection{Some basics from graph theory}
Now, we recall some basic facts from graph theory which will be used in studying binomial edge ideals. Let $G$ be a simple graph with the vertex set $V(G)$ and the edge set $E(G)$. A subgraph $G_{sub}$ of $G$ is said to be an \textit{induced subgraph} if for all $u, v \in V(G_{sub})$ such that $\{u,v\} \in E(G)$ implies that $\{u,v\} \in E(G_{sub})$. If $v\in V(G)$, then $G \setminus v$ denotes the induced subgraph on the vertex set $V(G) \setminus \{v\}$ and for $T\subseteq V(G)$, $G\setminus T$ denotes the induced subgraph on the vertex set $V(G)\setminus T$. For any subset $T \subseteq V(G)$, denote $\omega(T)$ to be the number of connected components of $G \setminus T$. A subset $T\subseteq V(G)$ is called a \emph{cutset} in $G$ if either $T=\emptyset$ or $T\neq \emptyset$ and $\omega(T\setminus \{v\})<\omega(T)$ for every vertex $v\in T$. We denote by $\mathcal{C}(G)$ the set of all cutsets for $G$.

For a vertex $v$ in $G$,
$N_G(v) := \{u \in V(G) :  \{u,v\} \in E(G)\}$ denotes the
\textit{neighborhood} of $v$ in $G$. The \textit{degree} of a vertex  $v$, denoted by $\deg_G(v)$, is
$|N_G(v)|$. 

We further recall the definition of cycle which will be used in Section \ref{Sec: CM} and more extensively in Section \ref{Sec: Prop}. A \textit{cycle} is a connected graph $G$ with $\deg_G(v)=2$ for all $v\in V(G)$.  We denote the cycle on $n$ vertices by $\mathbf{C}_n$ for $n\geq 3$. In particular, $\mathbf{C}_3$ is triangle and $\mathbf{C}_4$ is square. Note that by a cycle we mean it is an induced cycle.


\subsection{Regularity and projective dimension}

We now recall the definition of two important homological invariants which can be computed directly from the Betti table. One of them is the Castelnuovo-Mumford regularity which can be used to measure the complexity of the structure of a graded module over a polynomial ring or a coherent sheaf on a projective space. In \cite{Mumford66}, Mumford defined the regularity of a coherent sheaf on projective space and generalized the ideas of Castelnuovo. Later on, Eisenbud and Goto \cite{EG-84} extended the definition of regularity for modules. Let $M$ be a finitely generated graded $S$-module and $\beta_{i,i+j}(M)$ the graded Betti numbers. The \textit{projective dimension} of $M$ is defined as $\pd(M):=\max\{i : \beta_{i,i+j}(M) \neq 0 \text{ for some } j\}$
 and the \textit{Castelnuovo-Mumford regularity} (or simply, \textit{regularity}) of $M$ is defined as 
 $\reg(M):=\max \{j : \beta_{i,i+j}(M) \neq 0 \text{ for some } i\}.$ 
In Section \ref{Regu}, we give bounds on the regularity of powers of binomial edge ideals of Levi graphs associated to certain plane curve arrangements.
 
\section{Cohen-Macaulay binomial edge ideals}\label{Sec: CM}

Firstly, we recall some known results on binomial edge ideals.
Let $G$ be a simple graph and $T\in \mathcal{C}(G)$ be a cutset in $G$. Let $G_1,\cdots,G_{\omega(T)}$ be the connected components of $G\setminus T$ and for every $i$, $\tilde{G_i}$ denotes the
complete graph on the vertex set $V(G_i)$. Moreover, we set $P_T(G) := \big(\underset{i\in T} \cup \{x_i,y_i\}, J_{\tilde{G_1}},\cdots, J_{\tilde{G}_{\omega(T)}}\big)$. 

Using the description of $P_T(G),$ Herzog et al. \cite{HHHKR} proved that the binomial edge ideal $J_G =  \underset{T \in \mathcal{C}(G)}\cap
P_T(G) \subseteq S=\mathbb{K}[x_i,y_i \mid i \in V(G)]$ which, in particular, implies that $J_G$ is a radical ideal. As a consequence, they also obtained the following formula for the dimension of $S/J_G$:
\begin{theorem}(\cite[Corollaries 3.3 and 3.9]{HHHKR})\label{thm: dimension formula}
    Let $G$ be a simple graph with the vertex set $V(G)$. Then 
    \[
  \dim(S/J_G)=\max \{(|V(G)|-|T|)+\omega(T):T\in \mathcal{C}(G)\}.
    \] In particular, if $G$ is connected, then $\dim(S/J_G)\geq |V(G)|+1$.
\end{theorem}
Next, we recollect the notion of Cohen-Macaulay rings.
\begin{definition}
    Let $I\subseteq S$ be a graded ideal. Then the quotient ring $S/I$ is called Cohen–Macaulay if $\depth(S/I) =\dim(S/I)$. The Cohen-Macaulay defect, denoted by $\cmdef(S/I)$, is defined by $\dim(S/I)-\depth(S/I).$ We say a graph $G$ is Cohen-Macaulay when $S/J_G$ is a Cohen-Macaulay ring.
\end{definition}

Using the dimension computation, we show the Cohen-Macaulayness of the Levi graph associated to the point-line configuration $\mathcal{L}$ in Example \ref{exam: point-line-config}. We also use Theorem \ref{thm: dimension formula} to study the dimension of binomial edge ideals of Levi graphs associated to the Hirzebruch quasi-pencils.

We now give an example of a point-line configuration whose associated Levi graph $G$ is Cohen-Macaulay. 
\begin{example}\label{exam: point-line-config}
    Consider a point-line configuration $\mathcal{L}$ in the plane consisting of $2$ points $p_1, p_2$ and $2$ lines $\ell_1, \ell_2$ such that there is exactly one double intersection point $p_1$ and the point $p_2$ lies on one of the lines (which is not an intersection point!). Let $G$ be the Levi graph with $V(G)=\{x_1,x_2,y_1,y_2\}$, where $x_i$  represents the point $p_i$ and $y_i$ represents the line $\ell_i$ for $i=1,2.$
     Then the edge set is $E(G)=\{\{x_1, y_1\},\{x_1,y_2\},\{x_2, y_2\}\}.$ One can observe that $G$ is the path $P_4$. Therefore, by the proof of \cite[Corollary 2.3]{MM13}, $J_G$ is complete intersection; hence, $S/J_G$ is Cohen-Macaulay. In fact, by using Macaulay2 \cite{M2}, we can compute the Betti table of $S/J_G$ as follows:
    \begin{verbatim}
            0 1 2 3
     total: 1 3 3 1
         0: 1 . . .
         1: . 3 . .
         2: . . 3 .
         3: . . . 1
    \end{verbatim}
    By the above Betti table, we have $\pd(S/J_G)=3$ and hence, by the Auslander-Buchsbaum formula, $\depth(S/J_G)=5$. It can be shown that the set of all \textit{cutsets} is given by 
$$\mathcal{C}(G)=\{\emptyset, \{x_1\},\{y_2\}\}.$$
 Observe that $\omega(\{x_1\})=\omega(\{y_2\})=2$.
 Therefore by Theorem \ref{thm: dimension formula}, $\dim(S/J_G)=5$. Hence $S/J_G$ is Cohen-Macaulay.
\end{example} 

However, in Sections \ref{subs-hir} and \ref{sec: not-Cohen-Macaulay}, we show cases of Levi graphs associated to plane curve arrangements that are not Cohen-Macaulay.

\subsection{Binomial edge ideals associated with Hirzebruch quasi-pencils}\label{subs-hir}
In this section, we study the binomial edge ideals of Levi graphs associated with Hirzebruch quasi-pencils. Hirzebruch quasi-pencils are line arrangements where the total number of singular points $s$ is equal to the number of lines $k$.
In this context, de Bruijn and Erd\"{o}s \cite{Erdos-Bru} proved a complete classification of
line arrangements with $s = k$.
\begin{theorem}(de Bruijn and Erd\"{o}s)
    Let $\mathcal{L}\subseteq \mathbb{P}_{\KB}^2$ be an arrangement of $k \geq 3$ lines in the planes with $s$ intersection points such that $t_k=0$. Then $s\geq k$ and the equality holds if and only if $\LL$ is one of the following:
    \begin{enumerate}
        \item  A Hirzebruch quasi-pencil with the intersection points satisfying $t_{k-1}=1$ and $t_2=k-1$;
        \item A finite projective plane arrangement consisting of $q^2+q+1$ points and $q^2+q+1$ lines, where $q=p^n$ for some prime $p$.
    \end{enumerate}
\end{theorem}

We now describe explicitly the Levi graph associated to the Hirzebruch quasi-pencil $\mathcal{L}\subseteq \mathbb{P}_{\mathbb{C}}^2$ with $k\geq3$ lines $\ell_1,\dots,\ell_k$ and $k\geq 3$ intersection points $p_1,\dots,p_k$. For $k\geq 3$, let $G_k$ be the associated Levi graph of $\LL$ with $V(G_k)=\{x_1,\dots,x_k,y_1,\dots,y_k\},$ where each vertex $x_i$ corresponds to the intersection point $p_i$ and each vertex $y_j$ corresponds to the line $\ell_j$. After rearranging, we can observe that $G_3$ is a cycle $C:x_1,y_2,x_3,y_3,x_2,y_1,x_1$ of length $6$. Now we construct $G_k$ for $k\geq 4$ inductively from $G_3$ by adding vertices and edges. For $k\geq 4$, we have $t_{k-1}=1$ and $t_2=k-1$. Therefore, there are $k-1$ vertices of degree $2$ in $G_k$ corresponding to the intersection points and the vertex corresponding to the remaining intersection point has degree $k-1$ in $G_k$. We fix the vertex $x_2$ such that $\deg_{G_k}(x_2)=k-1$. Also, there are $k-1$ vertices of degree $2$ in $G_k$ corresponding to the lines and the vertex corresponding to the remaining line has degree $k-1$ in $G_k$. We fix the vertex $y_2$ such that $\deg_{G_k}(y_2)=k-1$. For $k=4$, we obtain the Levi graph $G_4$ by adding two vertices $x_4,y_4$ and edges $\{\{x_2,y_4\},\{y_2,x_4\},\{y_4,x_4\}\}$ to $G_3$ so that $\deg_{G_4}(x_2)=\deg_{G_4}(y_2)=3$. Observe that $G_4$ consists of two cycles 
$$C:x_1,y_2,x_3,y_3,x_2,y_1,x_1 ; \hspace{2mm} C':x_2,y_4,x_4,y_2,x_3,y_3,x_2, \hspace{2mm} \text{where} \hspace{2mm} V(C)\cap V(C')=\{x_2,y_2,x_3,y_3\}.$$ Inductively, we construct $G_k$  by adding vertices $x_k,y_k$ and edges $\{\{x_2,y_k\},\{y_2,x_k\}, \{y_k,x_k\}\}$ to $G_{k-1}.$ Therefore, $G_k$ is a graph with vertex set $V(G_k)=\{x_1,\dots,x_k,y_1,\dots,y_k\}$ and edge set $E(G_k)=\{\{x_2,y_i\}:1\leq i\leq k, i\neq 2\}\cup \{\{x_j,y_2\}:1\leq j\leq k, j\neq 2\}\cup \{\{x_m,y_m\}:1\leq m\leq k, m\neq 2\}.$ (see Figure \ref{fig: quasi-pencil}.)

			\begin{figure}[H]
				\begin{tikzpicture}[scale=1.5]
    \draw [line width=2pt] (-4,3)-- (-4,2);
\draw [line width=2pt] (-4,2)-- (-4,1);
\draw [line width=2pt] (-4,1)-- (-4,0);
\draw [line width=2pt] (-4,3)-- (-5,2);
\draw [line width=2pt] (-5,2)-- (-5,1);
\draw [line width=2pt] (-5,1)-- (-4,0);
\draw [line width=2pt] (-4,3)-- (-3,2);
\draw [line width=2pt] (-3,2)-- (-3,1);
\draw [line width=2pt] (-3,1)-- (-4,0);
\draw [line width=2pt] (-4,3)-- (-2,2);
\draw [line width=2pt] (-2,2)-- (-2,1);
\draw [line width=2pt] (-2,1)-- (-4,0);
\draw [line width=2pt] (-4,3)-- (0,2);
\draw [line width=2pt] (0,2)-- (0,1);
\draw [line width=2pt] (0,1)-- (-4,0);
\draw [dotted ] (-0.2,1.5)-- (-1.8,1.5);
\begin{scriptsize}
\fill  (-4,3) circle (1.5pt);
\draw (-3.8,3.3) node {$x_2$};
\fill  (-4,2) circle (1.5pt);
\draw (-3.8,2) node {$y_3$};
\fill  (-4,1) circle (1.5pt);
\draw (-3.8,1) node {$x_3$};
\fill  (-4,0) circle (1.5pt);
\draw (-3.84,-0.3) node {$y_2$};
\fill  (-5,1) circle (1.5pt);
\draw (-4.7,1) node {$x_1$};
\fill  (-5,2) circle (1.5pt);
\draw (-4.7,2) node {$y_1$};
\fill  (-3,2) circle (1.5pt);
\fill  (-3,1) circle (1.5pt);
\fill  (-2,2) circle (1.5pt);
\draw (-1.8,2) node {$y_5$};
\draw (-1.8,1) node {$x_5$};
\fill  (-2,1) circle (1.5pt);
\fill  (0,2) circle (1.5pt);
\draw (0.3,2) node {$y_k$};
\fill  (0,1) circle (1.5pt);
\draw (0.3,1) node {$x_k$};
\draw (-2.8,1) node {$x_4$};
\draw (-2.8,2) node {$y_4$};
\end{scriptsize}

    \end{tikzpicture}
				\caption{Levi Graph $G_k$ associated to the Hirzebruch quasi-pencil $\mathcal{L}$}\label{fig: quasi-pencil}
			\end{figure}
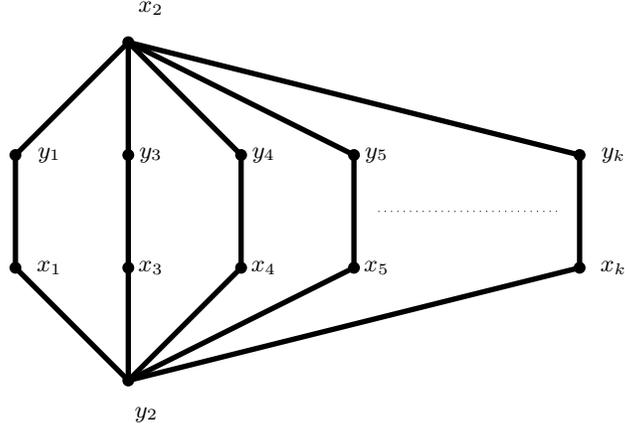

Now, we study the Cohen-Macaulayness of Levi graphs associated to Hirzebruch quasi-pencils.

Let $G=G_k$ for $k\geq 3$. In order to study the Cohen-Macaulayness of $S/J_G$, we first describe the set of cutsets of $G$ and hence obtain the dimension of $S/J_G$.
\begin{theorem}\label{thm: dim-Hirzebruch quasi-pencil}
Let $G=G_k$, $k\geq3$ be the Levi graph associated to the Hirzebruch quasi-pencil $\mathcal{L}\subseteq \mathbb{P}_{\mathbb{C}}^2.$ Let $V(G)=V_1\cup V_2$, where $V_1=\{x_1,\dots,x_k\}$ and $V_2=\{y_1,\dots,y_k\}$. Then $$\mathcal{C}(G)=\{\emptyset, \{x_2,y_2\},A,B,\{q_i:1\leq i\leq k, i\neq 2\}\},$$ where $A\subseteq V_1$ such that $x_2\in A$, $B\subseteq V_2$ such that $y_2\in B$ and either $q_i=x_i$ or $q_i=y_i$ for all $i \in \{1,3,4,\dots,k\}$.
Moreover, 

    $\dim(S/J_G)=
        \begin{cases}
            7 \quad \text{if } k=3 ;\\
            3k-3 \quad \text{if } k\geq 4.
        \end{cases}
        $
\end{theorem}
\begin{proof}
Let $T\in \mathcal{C}(G)$ be any non-empty cutset of $G$. Assume that $x_2\in T$. We now consider two cases: $y_2 \in T$ or $y_2 \notin T.$ If $y_2\in T$, then we claim that $T$ must be equal to $\{x_2,y_2\}$. If not, then for some $i \neq 2$, $q_i\in T$, where $q_i$ is either $x_i$ or $y_i.$ Then we see that  $\omega(T\setminus \{q_i\}) \geq \omega(T)$, which is a contradiction to the fact that $T\in \mathcal{C}(G)$.

If $y_2\notin T$ and $y_i\in T$ for some $i\neq 2$, then it can be observed that $\omega(T\setminus \{y_i\})\geq \omega(T)$, a contradiction. Therefore, if $x_2\in T$, then either $T=\{x_2,y_2\}$ or $T\subseteq V_1$. Similarly, one can show that if $y_2\in T$, then either $T=\{x_2,y_2\}$ or $T\subseteq V_2$. 

Suppose now $x_2\notin T$ and $y_2\notin T$. We note that if both $x_i$ and $y_i$ belong to $T$ for some $i \neq 2$, then $T$ is not a cutset of $G$. Therefore, for every $i\neq 2$, either $x_i\notin T$ or $y_i\notin T$. If for some $i\neq 2$, $x_i \notin T$ and $y_i\notin T$, then $G\setminus T$ contains the path: $x_2,y_i,x_i,y_2$, and hence $G\setminus T$ is connected. Therefore, $\omega(T)=1$. This implies that for any $v\in T$, $\omega(T\setminus \{v\})\geq \omega(T)$. This contradicts the fact that $T\in \mathcal{C}(G)$. Therefore, for all $i \neq 2$, either $x_i \in T$ or $y_i \in T$. Hence, the assertion for the set of all cutsets $\mathcal{C}(G)$ follows.

Note that $|V(G)|=2k$. If $T=\{x_2,y_2\}$, then $\omega(T)=k-1$. Also, if $T=A$ for some $A\subseteq V_1$ with $x_2\in A$, then $\omega(T)=|T|$ and for $T=\{q_i:1\leq i\leq k, i\neq 2\}$, we have $\omega(T)=2$. Therefore, it follows from Theorem \ref{thm: dimension formula} that $\dim(S/J_G)=\max \{2k+1, 3k-3,2k, k+3\}$. Hence, $\dim(S/J_G)=2k+1=7$ if $k=3$, otherwise $\dim(S/J_G)=3k-3$.
\end{proof}
\begin{theorem}
    Let $G=G_k$, $k \geq 3$ be the Levi graph associated to the Hirzebruch quasi-pencil $\mathcal{L}\subseteq \mathbb{P}_{\mathbb{C}}^2.$ Then $S/J_G$ is never Cohen-Macaulay. Moreover, if $k=3$, then $\cmdef(S/J_G)=1$ and $\cmdef(S/J_G)\geq k-3$ for $k\geq 4$.
\end{theorem}
\begin{proof}
    For $k=3$, the assertion follows from \cite[Theorem 4.5]{ZafarBMS}. So assume that $k\geq 4$. It follows from \cite[Theorem B]{BN17} that $\depth(S/J_G)\leq |V(G)|+2-\kappa(G)$, where $\kappa(G) = \min\{|T| ~:~ T \subset V(G) \text{ and } G \setminus T \text{ is disconnected}\}$. Note that for any vertex $v\in V(G)$, the induced subgraph $G\setminus v$ is connected. Thus, we have $\kappa(G)\geq 2$. Therefore, $\depth(S/J_G)\leq |V(G)|=2k.$ By Theorem \ref{thm: dim-Hirzebruch quasi-pencil}, $\dim(S/J_G)=3k-3 >\depth(S/J_G)$ and $\dim(S/J_G)-\depth(S/J_G)\geq (3k-3)-2k=k-3$. Hence, $S/J_G$ is not Cohen-Macaulay with $\cmdef(S/J_G)\geq k-3$.
    
\end{proof}
 \subsection{Binomial edge ideals associated with curve arrangements}\label{sec: not-Cohen-Macaulay}
  In \cite{HH05}, Herzog and Hibi provided a combinatorial characterization of Cohen-Macaulay edge ideals associated with bipartite graphs. Using this characterization, Pokora and R\"{o}mer \cite{PR22} show that the edge ideals of Levi graphs associated to certain line arrangements are never Cohen-Macaulay. Motivated by their result, we study the Cohen-Macaulayness of binomial edge ideals of Levi graphs of various curve arrangements in the complex projective plane. 
  
  In our study, we use the characterization of Cohen-Macaulay binomial edge ideals of bipartite graphs given by Bolognini, Macchia and Strazzanti in \cite{DAV}. We need the following combinatorial property for Cohen-Macaulay bipartite graphs which will be used in the subsequent theorems. 

 \begin{lemma}\label{lem:deg-of-vertex}
 Let $G$ be a bipartite graph such that $S/J_G$ is Cohen-Macaulay. Then there exists a leaf in $G$ i.e., there exists $v\in V(G)$ such that $\deg_G(v)=1$.
 \end{lemma}
 \begin{proof}
Let $G$ be a bipartite graph such that $S/J_G$ is Cohen-Macaulay. Hence, $S/J_G$ is unmixed.
 By \cite[Proposition 2.3]{DAV}, $G$ has exactly two leaves. As any leaf has degree $1$, the assertion follows.
 \end{proof}
 
 \begin{theorem}\label{thm:not CM d arr}
     Let $\mathcal{C} = \{ C_1,\dots , C_k\}$ be a $d$-arrangement of $k\geq 3$ curves in $\mathbb{P}^2_{\mathbb{C}}$ with $s$ intersection points.
Let  $S=\mathbb{K}[x_v,y_v:v\in V(G)]$ and $J_G$ be the associated binomial edge ideal determined by the Levi graph $G$ of $\mathcal{C}$. Then $S/J_G$ is never Cohen-Macaulay.
 \end{theorem}
\begin{proof}
    Let $G$ be the Levi graph associated with the $d$-arrangement $\mathcal{C}$ with $V(G)=V_1\cup V_2$, where $V_1=\{x_1,\dots,x_s\}$ corresponds to the intersection points of $\mathcal{C}$ and $V_2=\{y_1,\dots,y_k\}$ corresponds
to the curves in $\mathcal{C}$. Note that $\deg_G(x_i)\geq 2$ for all $i.$ First, we assume that $\mathcal{C}$ is a pencil of $k$ lines i.e., $t_k=1$ and $d=1$. In this case, $G$ is a star graph $K_{1,k}$ with the root $x_1$ corresponding to the intersection point. Since $G\setminus x_1$ is a graph consisting of $k$ isolated vertices $\{y_1,\dots,y_k\}$, the vertex $x_1$ is a cut vertex and by Theorem \ref{thm: dimension formula}, $\dim(S/J_G)\geq 2k>|V(G)|+1=k+2$ as $k\geq 3$. Hence, by \cite[Corollary 3.4]{HHHKR}, $S/J_G$ is not Cohen-Macaulay. 

To study the other cases, we use the following combinatorial count \eqref{comb count}:\\
For every $C_i\in \mathcal{C}$, 
\begin{align}\label{eqn-1}
    d^2(k-1)=\sum_{p\in \sing(\mathcal{C})\cap C_i}(m_p-1),
\end{align}
    where $m_{p}$ denotes the multiplicity i.e., the number of curves from $\mathcal{C}$ passing through $p\in \sing(\mathcal{C})$.
    
We now suppose $\mathcal{C}$ is a line arrangement with $t_k=0$. If $C_i$ has only one intersection point $p$, then $m_{p}=k$ which contradicts the fact that $t_k=0$. Therefore, for every $C_i\in \mathcal{C}$, there are at least two intersection points from $\sing(\mathcal{C})$ and hence $\deg_G(y_i)\geq 2$ for all $i$.

 Lastly, we assume that $\mathcal{C}$ is a $d$-arrangement with $d\geq 2$. Since $m_p\leq k$ for all $p \in \sing(\mathcal{C})$, it follows from \eqref{eqn-1} that there are at least $d^2$ intersection points in $C_i$ i.e., $\deg_G(y_i)\geq d^2>2$ for all $i.$ So, in both the cases, there is no vertex $v$ in $G$ such that $\deg_G(v)=1$. Hence, by Lemma \ref{lem:deg-of-vertex}, $S/J_G$ is not Cohen-Macaulay.
\end{proof}

\begin{theorem}\label{thm:not CM conic-line}
    Let $\mathcal{CL}=\{\ell_1,\dots,\ell_n,C_1,\dots,C_k\}\subseteq \mathbb{P}^2_{\mathbb{C}}$ be a  conic-line arrangement with $s$ intersection points. Let $S_H=\mathbb{K}[x_v,y_v:v\in V(H)]$ and  $J_H$ be the associated binomial edge ideal determined by the Levi graph $H$ of $\mathcal{CL}$. Then $S_H/J_H$ is never Cohen-Macaulay.
\end{theorem}
\begin{proof}
    Let $H$ be the Levi graph associated with the conic-line arrangement $\mathcal{CL}$ with $V(H)=W_1\cup W_2$, where $W_1=\{x_1,\dots,x_s\}$ corresponds to the intersection points of $\mathcal{CL}$ and $W_2=\{y_1,\dots,y_n,y_{n+1},\dots,y_{n+k}\}$,  where $y_i$ corresponds to the line $\ell_i$, when $1 \leq i \leq n$ and $y_{n+j}$ corresponds to the conic $C_j$, when $1\leq j \leq k$ in $\mathcal{CL}$. Since each intersection point has multiplicity at least $2$, $\deg_H(x_i)\geq 2$. It can be observed that for every $\ell_i\in \mathcal{CL}$, 
\begin{align}\label{eqn-2}
    (n-1)+2k=\sum_{p\in \sing(\mathcal{CL})\cap \ell_i}(m_p-1) \text{ and }
\end{align}
for every $C_j\in \mathcal{CL}$,
\begin{align}\label{eqn-3}
    2n+4(k-1)=\sum_{p\in \sing(\mathcal{CL})\cap C_j}(m_p-1).
\end{align}
For every intersection point $p$, we have $m_p\leq n+k$. If $\ell_i$ has only one intersection point, then $m_p-1\leq n+k-1<(n-1)+2k$. Also if $C_j$ has only one intersection point, then $m_p-1\leq n+k-1<2n+4(k-1)$ as $n\geq 1$ and $k\geq 1$. This implies that there are at least two intersection points for every $\ell_i$ and $C_j$. Therefore, $\deg_H(y_i)\geq 2$ for all $1\leq i\leq n$ and $\deg_H(y_{n+j})\geq 2$ for all $1\leq j\leq k$. So, there is no vertex $v$ in $H$ such that $\deg_H(v)=1$. Hence, by Lemma \ref{lem:deg-of-vertex}, $S_H/J_H$ is not Cohen-Macaulay.
\end{proof}

\section{Certain properties of Levi graphs associated with curve arrangements}\label{Sec: Prop}
In this section, we study the existence of certain induced cycles in the Levi graphs of some plane curve arrangements and use this information in Section \ref{Regu} to obtain bounds for the regularity of powers of the corresponding binomial edge ideals.

Let $\mathcal{C} = \{C_1,\dots , C_k\} \subset \mathbb{P}^2_{\mathbb{C}}$ be a $d$-arrangement and $G$ be the associated Levi graph. In what follows, we use the following notations for $q_1, \dots, q_i \in \{1, \dots ,k\}$ with $q_1 < q_2 < \dots < q_i$ and for distinct $i,j,r \in \{1,2, \ldots,k\}$:

\vspace{2mm}
\begin{enumerate}
\item $C_{q_1q_2\cdots q_i}$: Set of $i$-fold points in $\sing{(\mathcal{C})}$ where exactly $C_{q_1}, C_{q_2},\dots,C_{q_i}$ from $\mathcal{C}$ meet.\\

\item $\sing_{i}{(\mathcal{C})}$: Set of all $i$-fold points in  $\sing{(\mathcal{C})}$ i.e.,
$$\sing_{i}{(\mathcal{C})}= \bigcup_{\begin{smallmatrix} q_1,\dots ,q_i \in \{1, \dots, k\}  & \\  q_1 <  \cdots < q_i \end{smallmatrix}} C_{ q_1 q_2 \cdots q_i}, $$
and \[ t_i(\mathcal{C}) = |\sing_{i}{(\mathcal{C})}|=   \sum_{\begin{smallmatrix} q_1, \dots, q_i \in \{1, \dots, k \} & \\ q_1 < \cdots < q_i \end{smallmatrix}} |C_{q_1q_2 \cdots q_i}|. \]

 \item 
 $\sing^{r}_{i}{(\mathcal{C})}$: Set of all points in $\sing_{i}{(\mathcal{C})} \cap C_r$ i.e.,
\[ \sing^{r}_{i}{(\mathcal{C})}=\bigcup_{\begin{smallmatrix} q_1,\dots ,q_{i-1} \in \{1, \dots, k\} \setminus\{r\} & \\  q_1  < \cdots < q_{i-1} \end{smallmatrix}} C_{rq_1 q_2 \cdots q_{i-1}} . \]
\item
 $\sing^{jr}_{i}{(\mathcal{C})}$: Set of all points in $\sing_{i}{(\mathcal{C})} \cap C_j \cap C_r$ i.e.,
 \[ \sing^{jr}_{i}{(\mathcal{C})}=\bigcup_{\begin{smallmatrix} q_1,\dots ,q_{i-2} \in \{1, \dots, k\} \setminus\{j,r\} & \\  q_1 < \cdots < q_{i-2} \end{smallmatrix}} C_{jrq_1 q_2 \cdots q_{i-2}} .\]
\item  
 $\sing^{jr}_{i}{(\mathcal{C})}\setminus C_i$: Set of all points in $\sing^{jr}_{i}{(\mathcal{C})}$ that do not belong to $\sing(\mathcal{C}) \cap C_i.$\\
 \item 
  $\sing^{ijr}_{i}{(\mathcal{C})}$: Set of all points in $\sing_{i}{(\mathcal{C})} \cap C_i \cap C_j \cap C_r$ i.e.,
\[ \sing^{ijr}_{i}{(\mathcal{C})}=\bigcup_{\begin{smallmatrix} q_1,\dots ,q_{i-3} \in \{1, \dots, k\} \setminus\{i,j,r\} & \\  q_1 < \cdots < q_{i-3} \end{smallmatrix}} C_{ijrq_1 q_2 \cdots q_{i-3}} .\]  
  
\end{enumerate}

We note that in this case, any induced $\mathbf{C}_{6}$ in $G$ must contain three vertices corresponding to three intersection points and three vertices corresponding to three curves from $\mathcal{C}.$ In order to emphasize the choice of intersection points and curves appearing in the induced cycle, we write its vertices in terms of the points and curves directly (as in Figure \ref{fig: Cycle}, for example) instead of using the notations of $V(G)$ as in Definition \ref{df: d-arr}.


\begin{theorem}\label{thm: C_6-containment for d-arrangement}
     Let $\mathcal{C} = \{ C_1,\dots , C_k\}$ be a $d$-arrangement of $k\geq 3$ curves in $\mathbb{P}^2_{\mathbb{C}}$ with $t_k = 0$. Then the associated Levi graph $G$ of $\mathcal{C}$ has an induced  $\mathbf{C}_{6}$.
 \end{theorem}

 \begin{proof}
Let us consider three curves $C_i, C_j, C_r$ in $\mathcal{C}$ for distinct $i,j,r \in \{1,2, \ldots,k\}$.
 We divide the proof into two cases.\\

\noindent \textbf{Case I} : At least one of $C_{ij}, C_{jr}, C_{ri}$ is non-empty.
 
 Without loss of generality, assume that $C_{ij} \neq \emptyset$ and let $p_{ij} \in C_{ij}$. If both $C_{jr}$ and $C_{ri}$ are also non-empty, then we have $2$-fold points $p_{jr} \in C_{jr}$ and $p_{ri} \in C_{ri}$. Since  $p_{ij} \in C_{ij}$ is a $2$-fold point of $\sing(\mathcal{C}) \cap C_i \cap C_j$, it cannot be on $C_r$. Similarly, $p_{jr}$ cannot be on $C_i$ and $p_{ri}$ cannot be on $C_j$. Thus, $p_{ij}, p_{jr}, p_{ri}$ along with curves $C_i, C_j, C_r$ correspond to an induced $\mathbf{C}_6$ in $G$ as described in Figure \ref{fig: Cycle}.

 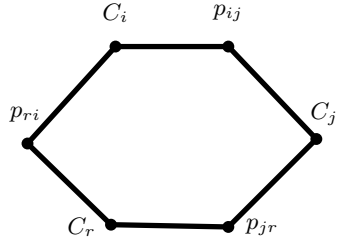
\begin{figure}[H]
   \centering
\begin{tikzpicture}[scale=1.5]

\draw [line width=2pt] (-2,2)-- (-1,2);
\draw [line width=2pt] (-1,2)-- (-0.22,1.18);
\draw [line width=2pt] (-2,2)-- (-2.78,1.14);
\draw [line width=2pt] (-2.78,1.14)-- (-2.04,0.42);
\draw [line width=2pt] (-2.04,0.42)-- (-1,0.4);
\draw [line width=2pt] (-1,0.4)-- (-0.22,1.18);
\begin{scriptsize}
\fill  (-2,2) circle (1.5pt);
\draw (-2,2.3) node {$C_i$};
\fill  (-1,2) circle (1.5pt);
\draw (-1,2.3) node {$p_{ij}$};
\fill  (-2.78,1.14) circle (1.5pt);
\draw (-2.8,1.4) node {$p_{ri}$};
\fill  (-2.04,0.42) circle (1.5pt);
\draw (-2.3,0.4) node {$C_r$};
\fill  (-0.22,1.18) circle (1.5pt);
\draw (-0.15,1.4) node {$C_j$};
\fill  (-1,0.4) circle (1.5pt);
\draw (-.7,0.4) node {$p_{jr}$};
\end{scriptsize}
\end{tikzpicture}\caption{Cycle of length $6$}\label{fig: Cycle}
\end{figure}

 We now assume that at least one of $C_{jr}$ or $ C_{ri}$ is empty. In this case, we prove the assertion by contradiction.
 
 Suppose $G$ has no induced $\mathbf{C}_{6}$ corresponding to the three curves $C_i, C_j$ and $C_r$.

  \vspace{2mm}
 \noindent \textbf{Subcase I}:
  Assume $C_{ij}, C_{jr} \neq \emptyset$ and $C_{ri} = \emptyset$. Let $p_{ij} \in C_{ij}$ and $p_{jr} \in C_{jr}$.

 Suppose there is a point $\bar{p}_{ri}$ in $\sing(\mathcal{C}) \cap C_r \cap C_i$ that does not belong to $\sing(\mathcal{C})\cap C_j$. Then, the $2$-fold points $p_{ij} \in C_{ij}, p_{jr} \in C_{jr}$ and $\bar{p}_{ri} \in \sing(\mathcal{C}) \cap C_r \cap C_i$ along with the three curves $C_i, C_j, C_r$ correspond to an induced $\mathbf{C}_6$ in $G$. 
 Since by hypothesis, $G$ does not have an induced $\mathbf{C}_6$ corresponding to $C_i, C_j$ and $C_r$, all points of $\sing(\mathcal{C}) \cap C_r \cap C_i$ must belong to $\sing(\mathcal{C}) \cap C_j$. In other words, all the points of $\sing(\mathcal{C}) \cap C_r \cap C_i$ come from the list: 
 \begin{equation}\label{list-1}
 \sing_3^{ijr}(\mathcal{C}), \sing_4^{ijr}(\mathcal{C}), \dots, \sing_{k-1}^{ijr}(\mathcal{C}).
 \end{equation}

     
Since $C_r$ and $C_i$ meet in $d^2$ many points,   
\begin{equation}\label{card-1}
|\sing(\mathcal{C}) \cap C_r \cap C_i| = |\sing_3^{ijr}(\mathcal{C})| + |\sing_4^{ijr}(\mathcal{C})| +\cdots+ |\sing_{k-1}^{ijr}(\mathcal{C})| = d^2.   
\end{equation}
Note that all points in list \eqref{list-1} also belong to $\sing(\mathcal{C}) \cap C_i \cap C_j$. Since $C_{ij} \neq \emptyset$, $p_{ij} \in C_{ij}$ does not come from \eqref{list-1}. Thus,
\begin{align}\label{card-2}
|\sing_3^{ijr}(\mathcal{C})| + |\sing_4^{ijr}(\mathcal{C})| +\cdots+ |\sing_{k-1}^{ijr}(\mathcal{C})| < |\sing(\mathcal{C}) \cap C_i \cap C_j| = d^2,
\end{align}
a contradiction.\\
The proof is similar in the case when $C_{jr}=\emptyset$ and $C_{ij},C_{ri} \neq \emptyset$. \\

 \noindent \textbf{Subcase II}:
Assume $C_{ij} \neq \emptyset$ and both $C_{jr}$ and $C_{ri}$ are empty. Let $p_{ij} \in C_{ij}$. 

Since $C_{jr} =\emptyset $, all points of $\sing(\mathcal{C}) \cap C_j \cap C_r$ have multiplicity at least three. We divide these intersection points into two sets, one for those that belong to $\sing(\mathcal{C}) \cap C_i$ and the other one for those that do not belong to $\sing(\mathcal{C}) \cap C_i$. The points of $\sing(\mathcal{C})\cap C_j \cap C_r$ that do not belong to $\sing(\mathcal{C}) \cap C_i$ must come from the following list:
\begin{equation}\label{list 1 setminus}
\sing_3^{jr}(\mathcal{C}) \setminus C_i, \sing_4^{jr}(\mathcal{C}) \setminus C_i,  \dots, \sing_{k-1}^{jr}(\mathcal{C}) \setminus C_i.
\end{equation}
If all the sets of list \eqref{list 1 setminus} are empty i.e., all the points of $\sing(\mathcal{C})\cap C_j \cap C_r$ belong to $\sing(\mathcal{C}) \cap C_i,$  all points of $\sing(\mathcal{C})\cap C_j \cap C_r$ come from list \eqref{list-1}. So, we can proceed as in the proof of \textbf{Subcase I} and arrive  at a contradiction.

Suppose there exist sets in list \eqref{list 1 setminus} that are non-empty, say, $\sing_a^{jr}(\mathcal{C}) \setminus C_i \neq \emptyset$ for  some $a \in \{3, \dots , k-1\}$.  Let $p_{jr}^{\star} \in \sing_a^{jr}(\mathcal{C}) \setminus C_i.$
 If there is a point $p_{ri}^{\star} \in (\sing(\mathcal{C})\cap C_r \cap C_i) \setminus (\sing(\mathcal{C}) \cap C_j)$, then $p_{ij}\in C_{ij}$, $p_{jr}^{\star}$ and $p_{ri}^{\star}$ along with the three curves $C_i, C_j, C_r$ correspond to an induced $\mathbf{C}_6$ in $G$, contrary to the hypothesis.

 So, all points of $\sing(\mathcal{C})\cap C_r \cap C_i$ also belong to  $\sing(\mathcal{C}) \cap C_j$ i.e., all points of $\sing(\mathcal{C})\cap C_r \cap C_i$ come from list \eqref{list-1}.
 
 But as in \textbf{Subcase I}, all points of list \eqref{list-1} also belong to $\sing(\mathcal{C}) \cap C_i \cap C_j$ and $p_{ij} \in C_{ij}$ does not come from list \eqref{list-1}. Hence, $|\sing(\mathcal{C})\cap C_r \cap C_i| < d^2$, a contradiction.



\vspace{3mm}

\noindent \textbf{Case II} : All of $C_{ij}, C_{jr}, C_{ir}$ are empty.

Let us consider the following lists :

\begin{align}
    \sing_3^{ij}(\mathcal{C}) \setminus C_r, \sing_4^{ij}(\mathcal{C}) \setminus C_r,\dots , \sing_{k-1}^{ij}(\mathcal{C}) \setminus C_r,\label{align 1} \\
    \space \space \space \sing_3^{jr}(\mathcal{C})\setminus C_i, \sing_4^{jr}(\mathcal{C})\setminus C_i,\dots, \sing_{k-1}^{jr}(\mathcal{C})\setminus C_i,\label{align 2} \\
    \sing_3^{ri}(\mathcal{C})\setminus C_j,  \sing_4^{ri}(\mathcal{C})\setminus C_j,\dots,  \sing_{k-1}^{ri}(\mathcal{C})\setminus C_j.\label{align 3}
 \end{align}
 \vspace{2mm}
 

 If at least one set in $\eqref{align 1},\eqref{align 2}$ or $\eqref{align 3}$ is non-empty, say for example, $ \sing_a^{ij}(\mathcal{C}) \setminus C_r \neq {\emptyset}$ for some $a \in \{3, \dots ,k-1\}$, then there exist $q_1, q_2,\dots , q_{a-2} \in \{1,\dots , k\} \setminus \{i, j, r\}$ with $q_1 < q_2 < \dots < q_{a-2}$ such that $C_{ijq_1q_2 \cdots q_{a-2}} \neq \emptyset$. 
 
 Note that in this case, both $C_{jr}$ and $C_{ri}$ are empty and $C_{ijq_1q_2 \cdots q_{a-2}} \neq \emptyset$. Let $p_{ij}^{\#} \in C_{ijq_1q_2 \cdots q_{a-2}}$. We prove by contradiction. Suppose $G$ has no induced $\mathbf{C}_6$ corresponding to $C_i, C_j$ and $C_r$. We run the same process as in \textbf{Subcase II}. We provide some details for the sake of completion.
 
Observe that all points of $\sing(\mathcal{C}) \cap C_j \cap C_r$ come from either list $\eqref{align 2}$ or $\eqref{list-1}$. If some sets of list $\eqref{align 2}$ are non-empty, say, $\sing_b^{jr}(\mathcal{C}) \setminus C_i \neq \emptyset$ for  some $b \in \{3, \dots , k-1\},$ choose a point $p_{jr}^{\#}  \in \sing_b^{jr}(\mathcal{C}) \setminus C_i$. If there is a point $p_{ri}^{\#} \in (\sing(\mathcal{C})\cap C_r \cap C_i) \setminus (\sing(\mathcal{C}) \cap C_j)$, then $p_{ij}^{\#}$, $p_{jr}^{\#}$ and $p_{ri}^{\#}$ along with the three curves $C_i, C_j, C_r$ correspond to an induced $\mathbf{C}_6$ in $G$, contrary to the hypothesis.
 
 So, all the points of $\sing({\mathcal{C}}) \cap C_r \cap C_i$ come from list $\eqref{list-1}$. Since $C_{ijq_1q_2 \cdots q_{a-2}} \neq \emptyset$ and $r \notin \{q_1, \dots, q_{a-2}\}$, $|\sing(\mathcal{C}) \cap C_r \cap C_i| < d^2$ as in \textbf{Subcase I}, a contradiction. Similarly, if all the sets of list $\eqref{align 2}$ are empty, then we arrive at a contradiction by showing that $|\sing(\mathcal{C}) \cap C_j \cap C_r| < d^2.$

Now, we are left with the case where all the sets of $\eqref{align 1},\eqref{align 2}$ and $\eqref{align 3}$ are empty. Hence, all points of $\sing(\mathcal{C}) \cap C_i \cap C_j$, $\sing(\mathcal{C}) \cap C_j \cap C_r$ and $\sing(\mathcal{C}) \cap C_r \cap C_i$ come from $\eqref{list-1}$.
 
 Next, we consider $C_i, C_j$ and $C_{r_1}$ for $r_1 \in \{1,\dots ,k\} \setminus \{i ,j ,r \}$ and proceed as above. If $C_{ir_1}$ or $C_{jr_1}$ is non-empty, then we again proceed as in \textbf{Case I}.
  If both $C_{ir_1}$ and $C_{jr_1}$ are non-empty, then we have $C_{jr_1},C_{ir_1} \neq \emptyset$ and $C_{ij} = \emptyset$, which is same as \textbf{Subcase I}. If $C_{ir_1}$ is non-empty and $C_{jr_1}$ is empty, then we have $C_{ir_1} \neq \emptyset, C_{jr_1}, C_{ij} = \emptyset$, which is same as \textbf{Subcase II}. In either case, $G$ has an induced $\mathbf{C}_6$ corresponding to $C_i, C_j$ and $C_{r_1}$.
 
So, we consider the case when $C_{ij}$, $C_{jr_1}$ and $C_{ir_1}$ are all empty and examine the following lists:

 \begin{align}
     \sing_3^{ij}(\mathcal{C}) \setminus C_{r_1}, \sing_4^{ij}(\mathcal{C}) \setminus C_{r_1},\dots , \sing_{k-1}^{ij}(\mathcal{C}) \setminus C_{r_1},\label{align 4} \\
    \space \space \space \sing_3^{jr_1}(\mathcal{C})\setminus C_i, \sing_4^{jr_1}(\mathcal{C})\setminus C_i,\dots, \sing_{k-1}^{jr_1}(\mathcal{C})\setminus C_i,\label{align 5} \\
    \sing_3^{r_1i}(\mathcal{C})\setminus C_j,  \sing_4^{r_1i}(\mathcal{C})\setminus C_j,\dots,  \sing_{k-1}^{r_1i}(\mathcal{C})\setminus C_j.\label{align 6}
 \end{align}
\vspace{2mm}

 If at least one set in $\eqref{align 4}$, $\eqref{align 5}$ or $\eqref{align 6}$ is non-empty, then we are done as in the case of $C_i, C_j$ and $C_r$. If not, all points of $\sing(\mathcal{C}) \cap C_i \cap C_j$, $\sing(\mathcal{C}) \cap C_j \cap C_{r_1}$ and $\sing(\mathcal{C}) \cap C_{r_1} \cap C_i$ come from the list

 \begin{align}\label{align 7}
     \sing_3^{ijr_1}(\mathcal{C}), \sing_4^{ijr_1}(\mathcal{C}),\dots, \sing_{k-1}^{ijr_1}(\mathcal{C}).
 \end{align}

Since all points of $\sing(\mathcal{C}) \cap C_i \cap C_j$ come from \eqref{align 7}, all $3$-fold points of $\sing(\mathcal{C}) \cap C_i \cap C_j$ come from $C_{ijr_1}$. So, $C_{ijr} = \emptyset$ i.e., $\sing^{ijr}_3(\mathcal{C}) = \emptyset$. As all points of $\sing(\mathcal{C}) \cap C_i \cap C_j$ also come from \eqref{list-1}, we observe that all $4$-fold points of $\sing(\mathcal{C}) \cap C_i \cap C_j$ come from $\sing_4^{ijr}(\mathcal{C}) \cap \sing_4^{ijr_1}(\mathcal{C}) = C_{ijrr_1}$. From this, we cannot immediately ensure the existence of an induced $\mathbf{C}_6$. Hence, we go onto the next step and consider $C_i, C_j$ and $C_{r_2}$ for $r_2 \in \{1, \dots, k\} \setminus \{i, j, r, r_1\}$ and proceed.
If we get an induced $\mathbf{C}_6$ corresponding to $C_i, C_j, C_{r_2}$, we stop. If not, as earlier, all points of $\sing(\mathcal{C}) \cap C_i \cap C_j$, $\sing(\mathcal{C}) \cap C_j \cap C_{r_2}$ and $ \sing(\mathcal{C}) \cap C_{r_2} \cap C_i$ come from the list

\begin{align}\label{align 8}
 \sing_4^{ijr_2}(\mathcal{C}), \sing_5^{ijr_2}(\mathcal{C}),\dots, \sing_{k-1}^{ijr_2}(\mathcal{C}).
 \end{align}

Here, in \eqref{align 8}, $\sing_3^{ijr_2}(\mathcal{C}) = C_{ijr_2} = \emptyset$, since all $3$-fold points of $\sing(\mathcal{C}) \cap C_i \cap C_j$ come from $C_{ijr_1}$ and $r_2 \neq r_1$. Also, from list \eqref{align 8}, we see that all $4$-fold points of $\sing(\mathcal{C}) \cap C_i \cap C_j$ come from $\sing_4^{ijr_2}(\mathcal{C}).$ Thus, $C_{ijrr_1} = \emptyset$ i.e., $\sing^{ijr}_4(\mathcal{C}) = \emptyset$. Similarly, we can see that all $5$-fold points of $\sing(\mathcal{C}) \cap C_i \cap C_j$ come from $C_{ijrr_1r_2}$.
\vspace{2mm}

Continuing the above process, if we get an induced $\mathbf{C}_6$ coming from $C_i, C_j, C_{r_m}$ for some $r_{m} \in \{1,\dots , k\} \setminus \{i ,j ,r,r_1, \dots, r_{m-1}\}$, we are done. If not, we get 
sequences like $\eqref{list-1}$, $\eqref{align 7}$, $\eqref{align 8}$ and so on. These sequences in each step reduce the number of non-empty sets in $\eqref{list-1}$ and finally makes all the sets in $\eqref{list-1}$ empty, a contradiction to the fact that all points of $\sing(\mathcal{C}) \cap C_i \cap C_j$ come from \eqref{list-1}.




 \end{proof}
 We now list down some observations on whether the conclusion of Theorem \ref{thm: C_6-containment for d-arrangement} holds or not when $t_k \neq 0.$
 
 \begin{remark}\label{no cycles lines}
     Let $\mathcal{C} = \{C_1, \dots, C_k\}$ be a $d$-arrangement of $k \geq 3$ curves with $t_k \neq 0$.\\
     If $d=1$ i.e., $\mathcal{C}$ is a pencil of lines, the associated Levi graph $G$ is a star graph that does not have an induced $\mathbf{C}_6.$\\
For $d=2$ i.e., when $\mathcal{C} = \{C_1, \dots, C_k\}$ is a $2$-arrangement of $k \geq 3$ conics with $t_k \neq 0$, we give an example below where the associated Levi graph $G$ does not have an induced $\mathbf{C}_6$.

 \vspace{2mm}
\begin{itemize}
\item Let $\mathcal{C}$ be a pencil of $k \geq 3$ smooth conics in $\mathbb{P}_\mathbb{C}^2$ i.e., an arrangement of conics satisfying  the following intersection data:  $t_k = 4$ and $t_i = 0$ for $i < k$. Since each conic passes through all $4$ points, the associated Levi graph $G$ does not have an induced $\mathbf{C}_6$.

\end{itemize}

But for $2$-arrangements with $t_k \neq 0$ we cannot always make the conclusion that the associated Levi graph $G$ does not have an induced $\mathbf{C}_6$. We discuss an example below of a $2$-arrangement with $t_k \neq 0$ whose associated Levi graph $G$ does have an induced $\mathbf{C}_6$.

\vspace{2mm}
\begin{itemize}
    
\item
Consider the Cremona-Klein configuration $\mathcal{K} = \{C_1, \dots, C_{21}\}$ of $21$ conics with the following data as in \cite[Section 4.1]{PT16}:
$$ t_3= 28, t_4 = 21, t_{21} = 3. $$
We have on each conic $C_i \in \mathcal{K}$, $4$ triple points, $4$ quadruple points and $3$ points with multiplicity $21$. To construct a $\mathbf{C}_6$ in $G$, choose $i_1 \in \{2, \dots, 21\}$ such that $C_1 \cap C_{i_1}$ contains a triple point and let $p_1 \in C_{1i_1j_1}$ be such an intersection point for some $j_1 \in \{2, \dots, 21\} \setminus \{i_1\}$. Now, let us choose $i_2 \in \{2, \dots, 21\} \setminus \{i_1,j_1\}$ and let $p_2$ be a triple or quadruple point in $C_{i_1} \cap C_{i_2}$. Since $C_1 \cap C_{i_1}$ consists of $3$ points of multiplicity $21$ and one triple point $p_1$, $p_2$ does not lie on $C_1$. Similarly, we can get $p_3 \in C_1 \cap C_{i_2}$ which does not lie on $C_{i_1}$. Therefore, $C_1, C_{i_1}, C_{i_2}$ along with intersection points $p_i$ for $i=1, 2, 3$ ensure the existence of an induced $\mathbf{C}_6$ in the associated Levi graph $G$. 
\end{itemize}
 \end{remark}

     


We now study the existence of an induced cycle of maximum length in Levi graphs associated to some $d$-arrangements.  

 \begin{theorem}\label{ind-cyc}
Let $\mathcal{C} = \{ C_1,\dots , C_k \}$ be a $d$-arrangement of $k \geq 4$ smooth curves of degree $d$ in $\mathbb{P}^2_{\mathbb{C}}$ with $t_2 \neq 0$ and $t_r =0$ for all $r>2$.
Then the associated Levi graph $G$ of $\mathcal{C}$ has an induced $\mathbf{C}_{2k}.$
 \end{theorem}
 \begin{proof}
 Since all intersection points of  $\mathcal{C}$ are double points, $C_{ij} \neq \emptyset$ for all $ i, j \in \{1, \dots, k \}$ with $i < j$. We consider points $p_{ii+1} \in C_{ii+1}$ for $i = 1, \dots, k-1$ and $p_{kk+1} \in C_{k1}$. Then $p_{ii+1}$ for $i=1, \dots, k$ along with curves $C_i$ for $i =1, \dots, k $ correspond to an induced $\mathbf{C}_{2k}$ in $G$.
 
 \end{proof}
 \vspace{2mm}
 
We are unable to give an example of a $d$-arrangement $\mathcal{C}$ of $k\geq 4$ curves with $t_r > 0$ for some $r > 2$ having an induced $\mathbf{C}_{2k}$ in the associated Levi graph $G$ of $\mathcal{C}$.\\

So, it is natural to ask the following question : 
     \begin{question}
         Does there exist a $d$-arrangement of $k\geq 4$ curves in $\mathbb{P}_\mathbb{C}^2$ with $t_r>0$ for some $r>2$ whose associated Levi graph $G$ has an induced cycle of maximum length?
     \end{question}

\vspace{2mm}

Now, we examine the existence of an induced $\mathbf{C}_{6}$ in the Levi graphs associated to conic-line arrangements with ordinary singularities in $\mathbb{P}_{\mathbb{C}}^2$. We use the same idea of the proof as in Theorem \ref{thm: C_6-containment for d-arrangement} to show the existence.

Let $\mathcal{CL} =\{\ell_1,\dots,\ell_n,C_1,\dots,C_k\}\subseteq \mathbb{P}_{\mathbb{C}}^2$ be a  conic-line arrangement in $\mathbb{P}_{\mathbb{C}}^2$ and $H$ be the associated Levi graph. We denote for $s_1 ,\dots, s_i \in \{1,\dots ,n\}$ with $s_1 <\dots<s_i$, $q_1 ,\dots,q_{j} \in \{1,\dots,k\}$ with $q_1 < \dots< q_j$ and for distinct $i,j,r \in \{1,2, \ldots,k\}$, the following notations:

\vspace{2mm}
\begin{enumerate}
    \item $\ell_{s_1s_2\cdots s_i}C_{q_1q_2\cdots q_j}$: Set of $(i+j)$-fold points in $\sing(\mathcal{CL})$ where exactly $\ell_{s_1}, \ell_{s_2}, \cdots, \ell_{s_i}$ and $ C_{q_1},C_{q_2}, \cdots, C_{q_j}$ from $\mathcal{CL}$ meet.\\
    
    \item $\sing_{i+j}{(\mathcal{CL})}$: Set of all $(i+j)$-fold points in  $\sing{(\mathcal{CL})}$ where $i$ lines from $\mathcal{CL}$ and $j$ conics from $\mathcal{CL}$ meet i.e.,
\[\sing_{i+j}{(\mathcal{CL})}= \bigcup_{\begin{smallmatrix}    q_1 < \cdots < q_j & \\ s_1 <\dots<s_i\end{smallmatrix}}\ell_{s_1s_2\cdots s_i}C_{q_1q_2\cdots q_j}.\]  
    \item  $\sing^{ijr}_{i+j}{(\mathcal{CL})}$: Set of all points in $\sing_{i+j}{(\mathcal{CL})} \cap C_i \cap C_j \cap C_r$ i.e.,
\[ \sing^{ijr}_{i+j}{(\mathcal{CL})}= \bigcup_{\begin{smallmatrix}  q_1,\dots ,q_{j-3} \in \{1, \dots, k\} \setminus\{i,j,r\} & \\  q_1  < \cdots < q_{j-3} & \\ s_1 <\dots<s_i\end{smallmatrix}}\ell_{s_1s_2\cdots s_i}C_{ijr q_1q_2\cdots q_{j-3}}.\]  
\end{enumerate}



As in the case of $d$-arrangements, we write the vertices appearing in the induced cycle in terms of the intersection points and curves directly (as in Figure \ref{fig: Cycle2}, for example), instead of using the notations of $V(H)$ as in Definition \ref{df: conic line-arr}.

\begin{theorem}\label{thm: C_6-containment for conic-line arrangement}
  Let $\mathcal{CL} =\{\ell_1,\dots,\ell_n,C_1,\dots,C_k\}\subseteq \mathbb{P}_{\mathbb{C}}^2$ be a  conic-line arrangement in $\mathbb{P}_{\mathbb{C}}^2$ such that $t_{n+k} = 0$ with $n\geq3$ and $k\geq3.$ Then the associated Levi graph $H$ contains an induced $\mathbf{C}_{6}$.
 
   \end{theorem}
   \begin{proof}
  We divide the proof in two cases based on whether $\sing(\mathcal{CL}) \cap C_1 \cap \dots \cap C_k$ is non-empty or not. We deal with the case of $\sing(\mathcal{CL}) \cap C_1 \cap \dots \cap C_k = \emptyset$ in \textbf{Case I} and with the case of $\sing(\mathcal{CL}) \cap C_1 \cap \dots \cap C_k \neq \emptyset$ in \textbf{Case II}. We further divide \textbf{Case II} into two subcases depending on whether $\sing(\mathcal{CL}) \cap \ell_1 \cap \dots \cap \ell_n$ is empty or non-empty.
   
\vspace{2mm}

   \noindent \textbf{Case I:} All $k$ conics $C_1, \dots, C_k $ in $\mathcal{CL}$ do not pass through a single point.
   
 The proof works similarly to Theorem \ref{thm: C_6-containment for d-arrangement} and $H$ contains an induced $\mathbf{C}_6$ corresponding to three conics. We give some details of the proof and leave the remaining parts to the reader.
    
    Let us consider three conics $C_i, C_j$ and $C_r$ in $\mathcal{CL}$ for distinct $i,j,r \in \{1,2, \ldots,k\}$. We first assume that $C_{ij} \neq \emptyset$. If both $C_{jr}$ and $C_{ri}$ are also non-empty, then we get an induced $\mathbf{C}_6$ and we are done. If not, at least one of $C_{jr}$ and $C_{ri}$ is empty. Let $C_{ri} = \emptyset$ and $C_{jr}\neq \emptyset$. We prove the assertion by contradiction as in Theorem \ref{thm: C_6-containment for d-arrangement}. Suppose the associated Levi graph $H$ of $\mathcal{CL}$ has no induced $\mathbf{C}_{6}$ corresponding to $\{C_i, C_j,C_r\}$. If there is a point 
 $\bar{p}_{ri}$ in $\sing(\mathcal{CL}) \cap C_r \cap C_i$ that does not belong to $\sing(\mathcal{CL})\cap C_j$, then $p_{ij} \in C_{ij} , p_{jr} \in C_{jr}$ and $\bar{p}_{ri} \in \sing(\mathcal{CL}) \cap C_r \cap C_i$ along with curves $C_i, C_j, C_r$ correspond to an induced $\mathbf{C}_{6}$ in $H$- a contradiction. So, all points of $\sing(\mathcal{CL}) \cap C_r\cap C_i$ come from the following lists:
        \begin{align*}
        \sing_{0+3}^{ijr}(\mathcal{CL}), \sing_{0+4}^{ijr}(\mathcal{CL}), \dots , \sing_{0+ (k-1)}^{ijr}(\mathcal{CL}),\\
        \sing_{1+3}^{ijr}(\mathcal{CL}),  \sing_{1+4}^{ijr}(\mathcal{CL}), \dots,  \sing_{1+ (k-1)}^{ijr}(\mathcal{CL}).
        \end{align*}
        
        \vspace{2mm}
        

        Since $C_{ij} \neq \emptyset,$ proceeding similarly as in the proof of \textbf{Subcase I} of Theorem \ref{thm: C_6-containment for d-arrangement}, we can conclude that $|\sing(\mathcal{CL}) \cap C_r \cap C_i| < 4$,  a contradiction. Other cases are also very similar to the proofs of respective parts in Theorem \ref{thm: C_6-containment for d-arrangement}. We omit the details here.
        \vspace{2mm}
        
\noindent\textbf{Case II}: All $k$ conics in $\mathcal{CL}$ pass through a single point.

Now, we divide this case into two further subcases, depending on whether all the lines in $\mathcal{CL}$ pass through a point or not.

\vspace{2mm}
\noindent\textbf{Subcase I}: All $n$ lines $\ell_1, \dots , \ell_n$ in $\mathcal{CL}$ do not pass through a single point.

In this case, we consider the three lines $\ell_i, \ell_j, \ell_r$ instead of the three curves $C_i, C_j, C_r$ in $\textbf{Case I}$ and proceed similarly.

\vspace{2mm}
\noindent\textbf{Subcase II}: All lines in $\mathcal{CL}$ pass through a point.

Let $\sing(\mathcal{CL}) \cap \ell_1 \cap \dots \cap \ell_n = {p_1}$. Then $t_{n+k} = 0$ in the hypothesis implies that all the conics cannot pass through $p_1$ i.e., $ p_1 \notin \sing(\mathcal{CL}) \cap C_1 \cap \dots \cap C_k.$ Then there exists $i \in \{1, \dots, k\}$ such that $p_1 \notin C_i$. Now choose a point $p_2 \in \sing(\mathcal{CL}) \cap C_i \cap\ell_1$ and a point $p_3 \in \sing(\mathcal{CL}) \cap C_i \cap \ell_2$. 
 Since $\sing(\mathcal{CL}) \cap \ell_1 \cap \dots \cap \ell_n = {p_1}$, $p_2 \notin \sing(\mathcal{CL}) \cap \ell_2$ and  $p_3 \notin \sing(\mathcal{CL}) \cap \ell_1$. Now, $p_1,p_2,p_3$ along with $\ell_1, \ell_2, C_i$ correspond to an induced $\mathbf{C}_6$ in $H$ as in Figure \ref{fig: Cycle2}.

 \begin{figure}[H]
   \centering
\begin{tikzpicture}[scale=1.5]

\draw [line width=2pt] (-2,2)-- (-1,2);
\draw [line width=2pt] (-1,2)-- (-0.22,1.18);
\draw [line width=2pt] (-2,2)-- (-2.78,1.14);
\draw [line width=2pt] (-2.78,1.14)-- (-2.04,0.42);
\draw [line width=2pt] (-2.04,0.42)-- (-1,0.4);
\draw [line width=2pt] (-1,0.4)-- (-0.22,1.18);
\begin{scriptsize}
\fill  (-2,2) circle (1.5pt);
\draw (-2,2.3) node {$\ell_1$};
\fill  (-1,2) circle (1.5pt);
\draw (-1,2.3) node {$p_1$};
\fill  (-2.78,1.14) circle (1.5pt);
\draw (-2.8,1.4) node {$p_2$};
\fill  (-2.04,0.42) circle (1.5pt);
\draw (-2.3,0.4) node {$C_i$};
\fill  (-0.22,1.18) circle (1.5pt);
\draw (-0.15,1.4) node {$\ell_2$};
\fill  (-1,0.4) circle (1.5pt);
\draw (-0.7,0.4) node {$p_3$};
\end{scriptsize}
\end{tikzpicture}
\caption{Cycle of length $6$}
\label{fig: Cycle2}
\end{figure}
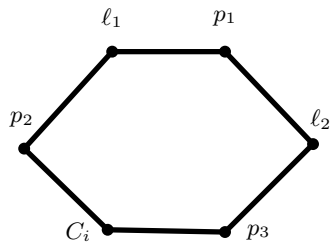


   \end{proof}

   We now give examples that illustrate the proof above.

   \begin{example}
       Let $\mathcal{CL}$ be the Hesse arrangement of $k=12$ conics and $n=9$ lines with the intersection points satisfying $t_2 = 72, t_5=12, t_9=9$. First, one can construct a $2$-arrangement $\mathcal{C}_h$ of $k=12$ conics based on Halphen pencil of index $2$ such that there are $9$ intersection points of multiplicity $8$ and $12$ intersection points of multiplicity $2$. Next take the dual Hesse arrangement $\mathcal{L}_h$ of $9$ lines whose all $12$ intersection points are of multiplicity $3$. Then we get the Hesse arrangement of conics and lines by making all the $12$ intersection points of $\mathcal{L}_h$ pass through the $12$ double points of $\mathcal{C}_h$. See \cite[Section $2.1$ and Example $6.6$]{PSCo} for details.

       It is clear from the above data that all the conics in $\mathcal{CL}$ do not pass through a single point i.e., $\sing(\mathcal{CL})\cap C_1\cap \cdots \cap C_{12}=\emptyset$. Hence, this is \textbf{Case I} of Theorem \ref{thm: C_6-containment for conic-line arrangement}. Therefore we wish to construct an induced $\mathbf{C}_6$ corresponding to three conics in $\mathcal{CL}$. Since, $\mathcal{C}_h$ is a sub $2$-arrangement of $\mathcal{CL}$, it is enough to construct an induced $\mathbf{C}_6$ in $G_h$, the associated Levi graph of $\mathcal{C}_h$.

       Let us denote the conics in $\mathcal{C}_h$ by $Q_i(\lambda)$ for $i =1, \ldots ,12$ and the intersection points of multiplicity $8$ by $p_j(\lambda)$ for $j=1, \ldots, 9$ as in \cite[Section $5$]{DPS23}. From the explicit description of conics given in \cite[Section $5$]{DPS23}, we see that 
       \begin{center}
            $p_1(\lambda) \in \sing(\mathcal{C}_h) \cap Q_1(\lambda) \cap Q_2(\lambda) \cap Q_3(\lambda) \cap Q_4(\lambda) \cap Q_5(\lambda) \cap Q_6(\lambda) \cap Q_7(\lambda) \cap Q_8(\lambda), $ \\
       $p_5(\lambda) \in \sing(\mathcal{C}_h) \cap Q_1(\lambda) \cap Q_3(\lambda) \cap Q_5(\lambda) \cap Q_6(\lambda) \cap Q_8(\lambda) \cap Q_9(\lambda) \cap Q_{11}(\lambda) \cap Q_{12}(\lambda),$\\
       $p_7(\lambda) \in \sing(\mathcal{C}_h) \cap Q_2(\lambda) \cap Q_3(\lambda) \cap Q_5(\lambda) \cap Q_7(\lambda) \cap Q_8(\lambda) \cap Q_9(\lambda) \cap Q_{10}(\lambda) \cap Q_{12}(\lambda).$
 \end{center}
 \vskip 2mm
       Then, $Q_1(\lambda)p_1(\lambda)Q_2(\lambda)p_7(\lambda)Q_9(\lambda)p_5(\lambda)Q_1(\lambda)$ corresponds to an induced $\mathbf{C}_6$ in $G_h$.
   \end{example}
   
\vskip 2mm
   \begin{example}
        Take four general points $P_1,P_2,P_3,P_4$ in $\mathbb{P}^2_{\mathbb{C}}$ and four conics $C_1,C_2,C_3,C_4$ in $\mathbb{P}^2_{\mathbb{C}}$ passing through them. Consider three lines, namely $\ell_1 = \overline{P_1P_4},$ $\ell_2 = \overline{P_2P_3}$ and $\ell_3 = \overline{P_2P_4}$. We have $t_2=1$, $t_5=2$ and $t_6=2$.

 Let $\ell_1$ and $\ell_2$ intersect in a point $P$. Then $P \notin \{P_1, P_2, P_3, P_4 \}$. So, $\ell_1P\ell_2P_2\ell_3P_4\ell_1$ is an induced $\mathbf{C}_6$ corresponding to $\ell_1, \ell_2$ and $\ell_3$. This is \textbf{Subcase I} of \textbf{Case II} in the proof of Theorem $5.5$.
 
Alternatively, we can construct an induced $\mathbf{C}_6$ corresponding to a conic and two lines. Let us consider $\ell_1, \ell_2$ and $C_1$. Then $C_1P_1\ell_1P\ell_2P_2C_1$ is an induced $\mathbf{C}_6$ corresponding to $\ell_1, \ell_2$ and $C_1$.\\
   \end{example}
 
\begin{remark}\label{conic-line geq 3}
In fact, Theorem \ref{thm: C_6-containment for conic-line arrangement} is true when $n+k \geq 4.$ If $n=2,$ $k >2$ or $k=2,$ $n>2,$ a similar process as in Theorem \ref{thm: C_6-containment for conic-line arrangement} may be applied to get an induced $\mathbf{C}_6$ in the associated Levi graph. If $n=k=2$, the proof is the same as \textbf{Subcase II} of Theorem \ref{thm: C_6-containment for conic-line arrangement}.
\end{remark}

\section{Bound on Castelnuovo–Mumford regularity }\label{Regu}
In this section, we investigate the regularity of binomial edge ideals of the Levi graphs coming from $d$-arrangements and conic-line arrangements. The first general lower bound for binomial edge ideals was obtained by Matsuda and Murai \cite[Corollary 2.3]{MM13}. They prove that $\reg(S/J_G)\geq \ell(G)$, where $\ell(G)$ is the length of a longest induced path in $G$. Later, Jayanthan, Kumar and the second author generalized the lower bound for powers of binomial edge ideals and proved the following:
\begin{theorem}(\cite[Corollary 3.4]{JAR2})\label{thm:lower-bound-on-regularity-of-power}
    Let $G$ be a connected graph and $\ell(G)$ be the length of a longest induced path of $G$. Then 
    $$\reg(S/J_G^t)\geq 2t+\ell(G)-2 \text{ for all }t\geq 1.$$
\end{theorem}
Using Theorems \ref{thm:lower-bound-on-regularity-of-power} and \ref{thm: C_6-containment for d-arrangement}, we conclude the following as the Levi graph of the $d$-arrangement $\mathcal{C}$ has an induced $\mathbf{C}_{6}$ which contains an induced path of length $4$:
\begin{corollary}\label{cor: lower-bound-d-arrangement}
    Let $\mathcal{C} = \{ C_1,\dots , C_k\}$ be a $d$-arrangement of $k\geq 3$ curves in $\mathbb{P}^2_{\mathbb{C}}$ with $t_k = 0$. Let $G$ be the associated  Levi graph. Then 
    $$\reg(S/J_G^t)\geq 2t+2 \text{ for all }t\geq 1.$$
\end{corollary}

We now give improved bounds for the regularity of powers of binomial edge ideals of Levi graphs associated to certain classes of arrangements.
\begin{corollary}\label{cor:lower bound-improved}
Let $\mathcal{C} = \{ C_1,\dots , C_k \}$ be a $d$-arrangement of $k \geq 4$ smooth curves of degree $d$ in $\mathbb{P}^2_{\mathbb{C}}$ with $t_2 \neq 0$ and $t_r =0$ for all $r>2$. Let $G$ be the associated Levi graph. Then
    $$\reg(S/J_G^t)\geq 2t+2k-4 \text{ for all }t\geq 1.$$
\end{corollary}
\begin{proof}
    By Theorem \ref{ind-cyc}, $G$ contains a cycle $\mathbf{C}_{2k}$ of length $2k$. As $\mathbf{C}_{2k}$ has an induced path of length $2k-2$, the assertion follows.
\end{proof}
Similarly, using Theorems \ref{thm:lower-bound-on-regularity-of-power}, \ref{thm: C_6-containment for conic-line arrangement} and Remark \ref{conic-line geq 3}, we have the following corollary:

\begin{corollary}\label{cor: lower-bound-conic-line-arrangement}
    Let $\mathcal{CL} =\{\ell_1,\dots,\ell_n,C_1,\dots,C_k\}\subseteq \mathbb{P}^2_{\mathbb{C}}$ be a  conic-line arrangement in $\mathbb{P}_{\mathbb{C}}^2$ such that $t_{n+k} = 0$ and $n+k \geq 4.$ Let $H$ be the associated Levi graph. Then 
    $$\reg(S_H/J_H^t)\geq 2t+2 \text{ for all }t\geq 1.$$
\end{corollary}
Now, we see an example of a Levi graph associated with a $d$-arrangement for which the regularity of powers of the binomial edge ideal attains the lower bound obtained in Corollary \ref{cor: lower-bound-d-arrangement}.
\begin{example}
    We consider the Levi graph $G=G_3$ associated with the Hirzebruch quasi-pencil of $k=3$ lines as in Section \ref{subs-hir}. In this case, $G$ is same as the cycle $\mathbf{C}_{6}$. Therefore, by \cite[Corollary 16]{ZZ13}, $\reg(S/J_G)=4$ and by \cite[Theorem 3.6]{JAR2}, $\reg(S/J_G^t)=2t+2$ for all $t\geq 2$.
\end{example}
In this context, we propose the following question:
\begin{question}
    Can we characterize Levi graphs $G$ associated with $d$-arrangements or conic-line arrangements such that $\reg(S/J_G^t)=2t+2$ for all $t\geq 1$?
\end{question}
       
\bibliographystyle{plain}
\bibliography{Reference}

\begin{thebibliography}{10}

\bibitem{BN17}
Arindam Banerjee and Luis N\'{u}\~{n}ez Betancourt.
\newblock Graph connectivity and binomial edge ideals.
\newblock {\em Proc. Amer. Math. Soc.}, 145(2):487--499, 2017.

\bibitem{BMRS22}
Davide Bolognini, Antonio Macchia, Giancarlo Rinaldo, and Francesco Strazzanti.
\newblock Cohen-{M}acaulay binomial edge ideals of small graphs.
\newblock {\em J. Algebra}, 638:189--213, 2024.

\bibitem{DAV}
Davide Bolognini, Antonio Macchia, and Francesco Strazzanti.
\newblock Binomial edge ideals of bipartite graphs.
\newblock {\em European J. Combin.}, 70:1--25, 2018.

\bibitem{DAV2}
Davide Bolognini, Antonio Macchia, and Francesco Strazzanti.
\newblock Cohen-{M}acaulay binomial edge ideals and accessible graphs.
\newblock {\em J. Algebraic Combin.}, 55(4):1139--1170, 2022.

\bibitem{Cox}
H.~S.~M. Coxeter.
\newblock Self-dual configurations and regular graphs.
\newblock {\em Bull. Amer. Math. Soc.}, 56:413--455, 1950.

\bibitem{Priya-survey}
Priya {Das}.
\newblock {Recent results on homological properties of binomial edge ideal of
  graphs}.
\newblock {\em arXiv e-prints}, page arXiv:2209.01201, September 2022.

\bibitem{Erdos-Bru}
N.~G. de~Bruijn and P.~Erd\"{o}s.
\newblock On a combinatorial problem.
\newblock {\em Nederl. Akad. Wetensch., Proc.}, 51:1277--1279 = Indagationes
  Math. 10, 421--423 (1948), 1948.

\bibitem{DPS23}
Alexandru {Dimca}, Piotr {Pokora}, and Gabriel {Sticlaru}.
\newblock {On the Alexander polynomials of conic-line arrangements}.
\newblock {\em arXiv e-prints}, page arXiv:2305.01450, May 2023.

\bibitem{EG-84}
David Eisenbud and Shiro Goto.
\newblock Linear free resolutions and minimal multiplicity.
\newblock {\em J. Algebra}, 88(1):89--133, 1984.

\bibitem{EHH-NMJ}
Viviana Ene, J\"urgen Herzog, and Takayuki Hibi.
\newblock Cohen-{M}acaulay binomial edge ideals.
\newblock {\em Nagoya Math. J.}, 204:57--68, 2011.

\bibitem{M2}
Daniel~R. Grayson and Michael~E. Stillman.
\newblock Macaulay2, a software system for research in algebraic geometry.
\newblock Available at \url{http://www.math.uiuc.edu/Macaulay2/}.

\bibitem{HH05}
J\"{u}rgen Herzog and Takayuki Hibi.
\newblock Distributive lattices, bipartite graphs and {A}lexander duality.
\newblock {\em J. Algebraic Combin.}, 22(3):289--302, 2005.

\bibitem{HHHKR}
J\"urgen Herzog, Takayuki Hibi, Freyja Hreinsd\'ottir, Thomas Kahle, and
  Johannes Rauh.
\newblock Binomial edge ideals and conditional independence statements.
\newblock {\em Adv. in Appl. Math.}, 45(3):317--333, 2010.

\bibitem{JAR2}
A.~V. Jayanthan, Arvind Kumar, and Rajib Sarkar.
\newblock Regularity of powers of quadratic sequences with applications to
  binomial ideals.
\newblock {\em J. Algebra}, 564:98--118, 2020.

\bibitem{KM-CA}
Dariush Kiani and Sara Saeedi~Madani.
\newblock Some {C}ohen-{M}acaulay and unmixed binomial edge ideals.
\newblock {\em Comm. Algebra}, 43(12):5434--5453, 2015.

\bibitem{MM13}
Kazunori Matsuda and Satoshi Murai.
\newblock Regularity bounds for binomial edge ideals.
\newblock {\em J. Commut. Algebra}, 5(1):141--149, 2013.

\bibitem{Mumford66}
David Mumford.
\newblock {\em Lectures on curves on an algebraic surface}, volume No. 59 of
  {\em Annals of Mathematics Studies}.
\newblock Princeton University Press, Princeton, NJ, 1966.
\newblock With a section by G. M. Bergman.

\bibitem{oh}
Masahiro Ohtani.
\newblock Graphs and ideals generated by some 2-minors.
\newblock {\em Comm. Algebra}, 39(3):905--917, 2011.

\bibitem{T092}
Peter Orlik and Hiroaki Terao.
\newblock {\em Arrangements of hyperplanes}, volume 300 of {\em Grundlehren der
  mathematischen Wissenschaften [Fundamental Principles of Mathematical
  Sciences]}.
\newblock Springer-Verlag, Berlin, 1992.

\bibitem{PRS}
Piotr Pokora, Xavier Roulleau, and Tomasz Szemberg.
\newblock Bounded negativity, {Harbourne} constants and transversal
  arrangements of curves.
\newblock {\em Annales de l'Institut Fourier}, 67(6):2719--2735, 2017.

\bibitem{PSCo}
Piotr Pokora and Tomasz Szemberg.
\newblock Conic-line arrangements in the complex projective plane.
\newblock {\em Discrete Comput. Geom.}, 69(4):1121--1138, 2023.

\bibitem{PT16}
Piotr Pokora and Halszka Tutaj-Gasi\'{n}ska.
\newblock Harbourne constants and conic configurations on the projective plane.
\newblock {\em Math. Nachr.}, 289(7):888--894, 2016.

\bibitem{PR22}
Pokora Pokora and Tim R\"{o}mer.
\newblock Algebraic properties of {L}evi graphs associated with curve
  arrangements.
\newblock {\em Res. Math. Sci.}, 9(2):Paper No. 30, 17, 2022.

\bibitem{Rinaldo-BMS}
Giancarlo Rinaldo.
\newblock Cohen-{M}acaulay binomial edge ideals of small deviation.
\newblock {\em Bull. Math. Soc. Sci. Math. Roumanie (N.S.)},
  56(104)(4):497--503, 2013.

\bibitem{Rinaldo-Cactus}
Giancarlo Rinaldo.
\newblock Cohen-{M}acaulay binomial edge ideals of cactus graphs.
\newblock {\em J. Algebra Appl.}, 18(4):1950072, 18, 2019.

\bibitem{RS-level}
Giancarlo Rinaldo and Rajib Sarkar.
\newblock Level and pseudo-{G}orenstein binomial edge ideals.
\newblock {\em J. Algebra}, 632:363--383, 2023.

\bibitem{Madani-survey}
Sara Saeedi~Madani.
\newblock Binomial edge ideals: a survey.
\newblock In {\em Multigraded algebra and applications}, volume 238 of {\em
  Springer Proc. Math. Stat.}, pages 83--94. Springer, Cham, 2018.

\bibitem{SS22}
Kamalesh {Saha} and Indranath {Sengupta}.
\newblock {Cohen-Macaulay Binomial edge ideals in terms of blocks with
  whiskers}.
\newblock {\em arXiv e-prints}, page arXiv:2203.04652, March 2022.

\bibitem{HT81}
Hiroaki Terao.
\newblock The exponents of a free hypersurface.
\newblock In {\em Singularities, {P}art 2 ({A}rcata, {C}alif., 1981)},
  volume~40 of {\em Proc. Sympos. Pure Math.}, pages 561--566. Amer. Math.
  Soc., Providence, RI, 1983.

\bibitem{ZafarBMS}
Sohail Zafar.
\newblock On approximately {C}ohen-{M}acaulay binomial edge ideal.
\newblock {\em Bull. Math. Soc. Sci. Math. Roumanie (N.S.)},
  55(103)(4):429--442, 2012.

\bibitem{ZZ13}
Sohail Zafar and Zohaib Zahid.
\newblock On the {B}etti numbers of some classes of binomial edge ideals.
\newblock {\em Electron. J. Combin.}, 20(4):Paper 37, 14, 2013.

\end{thebibliography}
\end{document}